\documentclass[a4paper,11pt,]{amsart}

\usepackage{xcolor}
\usepackage{tikz-cd}
\usepackage{amsmath,amssymb,amsthm}
\usepackage{here}
\usepackage[top=30truemm,bottom=30truemm,left=25truemm,right=25truemm]{geometry}
\usepackage[varg]{txfonts}

\numberwithin{equation}{section}

\theoremstyle{plain}
\newtheorem{theo}{Theorem}
\newtheorem*{mainthm*}{Main Theorem}
\theoremstyle{definition}
\newtheorem{defi}{Definition}[section]
\newtheorem{lemm}{Lemma}[section]

\newtheorem{coro}{Corollary}[section]
\newtheorem{fact}{Fact}[section]

\theoremstyle{remark}
\newtheorem{rema}{Remark}[section]

\begin{document}

\title[On discrete constant principal curvature surfaces]{On discrete constant principal curvature surfaces}

\author{Yutaro Kabata}
\address{School of Information and Data Sciences, Nagasaki University,
Bunkyocho 1-14, Nagasaki, 852-8131, Japan}
\email{kabata@nagasaki-u.ac.jp}

\author{Shigeki Matsutani}
\address{Faculty of Electrical, Information and Communication Engineering,
 Institute of Science and Engineering, Kanazawa University,
Kakuma, Kanazawa, 920-1192, Japan}
\email{s-matsutani@se.kanazawa-u.ac.jp}
 
\author{Yuta Ogata}
\address{
Department of Mathematics, Faculty of Science, Kyoto Sangyo University,
Motoyama, Kamigamo, Kita-ku, Kyoto-City, 603-8555, Japan}
\email{yogata@cc.kyoto-su.ac.jp} 

\keywords{discrete surface; constant principal curvature; constant bond length.}
\subjclass{%
Primary 53A70, Secondary 74A30}

\begin{abstract}
Recently, it is discovered that a certain class of nanocarbon materials has geometrical properties related to the discrete geometry, pre-constant discrete principal curvature \cite{KMNOO(??)} based on the discrete surface theory proposed on trivalent graphs by Kotani, Naito and Omori \cite{KNO(2017)}. In this paper, with the aim of an application to the nanocarbon materials, 
we will study discrete constant principal curvature (CPC) surfaces. Firstly, we developed the discrete surface theory on a full 3-ary oriented tree so that we define a discrete analogue of principal directions on them to investigate it. We also construct some interesting examples of discrete constant principal curvature surfaces, including discrete CPC tori.
\end{abstract}
\maketitle
\section{Introduction}
\quad\ The discrete differential geometry has been studied for three decades.
We note that there are several approaches to study discrete surfaces in different setups, from the viewpoints of the integrable system, material science, computer vision, and other applications.
In this paper, we develop the theory of discrete surfaces in $\mathbb{R}^3$ on a trivalent graph given by Kotani, Naito, and Omori \cite{KNO(2017)} to provide a geometrical basis for a geometrical analysis of a nanocarbon system, which was discovered and studied by Onoe \cite{OnoeNakayamaAonoHara(2003), ShimaYoshiokaOnoe(2009)} and whose form from the 
first-principles calculations was studied by Noda \cite{NodaOnoOhno(2015)}.
The geomtery of the nanocarbon systems, e.g., the fullenences, has been concerned in the material science. 
As in \cite{Terrones(1994), TerronesMackay(1994)}, Terrones and Mackay constructed a carbon network that is related to a discrete Schwarzian minimal surface, which is well-known as the typical example of triply periodic minimal surfaces. 
After their great works, many researchers currently are interested in the relation between nanomaterials science and mathematics (see \cite{NodaOnoOhno(2015),  OnoeNakayamaAonoHara(2003), ShimaYoshiokaOnoe(2009), Vassilev(2013)} also). 

Onoe \cite{OnoeNakayamaAonoHara(2003), ShimaYoshiokaOnoe(2009)} discovered an interesting nanocarbon materials, which is called peanut-shaped fullerene polymer (PSFP) or is simply referred to as a C${}_{60}$ polymer. 
The authors in this paper with Onoe and Noda \cite{KMNOO(??)} found that a certain class of the nanocarbon system, including C${}_{60}$, carbon nanotube, and the C${}_{60}$ polymer, has a novel discrete geometrical nature, i.e., the  discrete pre-constant principal curvature surfaces; some of them has the discrete constant principal curvature (CPC).
In \cite{KMNOO(??)}, it is revealed that the property is connected with the stability of the nanocarbon systems as a material science aspect.
However the geometrical property also shows the novel aspect of the discrete geometry. 
In other words, it motivates us to study discrete surfaces on a full 3-ary oriented tree, a discrete analogue of principal directions on them, and the property of constant bond length, etc (see Theorem \ref{theo:1} and \ref{theo:2} for detail).
In this paper and its sequel paper \cite{KMNOO(??)}, we show novel geometrical properties of the full 3-valent oriented graphs by introducing the discrete principal curvature and its line based on the theory by Kotani, Naito and Omori \cite{KNO(2017)}.

 Here we mention different two approaches on the discrete differential geometry in the previous works. Firstly, in \cite{BobenkoPinkall(1996), BobenkoPinkall(1999)}, Bobenko and Pinkall established a theory of discrete surfaces via quadrilaterals, using the technique of the integrable system. They considered the discretization of a special class of surfaces, called isothermic surfaces, and characterize discrete isothermic surfaces by cross ratios. Isothermic surfaces include famous classes of surfaces, for example minimal surfaces, CMC surfaces, surfaces of revolution, and so on. In \cite{BobenkoPinkall(1996), BobenkoPinkall(1999)}, they also defined the discrete mean and Gauss curvatures on the discrete surfaces via quadrilaterals by using the parallel meshes. They also gave a recipe of discrete isothermic minimal surfaces by using discrete holomorphic functions (i.e. isothermic map between complex planes $\mathbb{C}$), called a discrete Weierstrass representation. Even for the discrete CMC surfaces, recently there are several results and generalization in \cite{Tim(1999), HoffmannKobayashiYe(2022)}, etc. Especially, 
in \cite{Muller(2014), MullerWallner(2010)}, there are some studies related to the discrete surface theory via planar hexagonal meshes, by M{\"u}ller and Wallner. 
In the present paper, we will treat a full $3$-ary oriented tree, which may represent such hexagonal meshes as a particular case without the condition of planar; the state of planar is not necessary. See Figures \ref{fig:standardArmchair}, \ref{fig:KNOtube}, \ref{fig:Chiral}, \ref{fig:CPCtorus} and \ref{fig:FP5N} also.   

Secondly, we show the discrete surface theory related to triangles. 
In \cite{PinkallPolthier(1993)}, Pinkall and Polthier defined the mean curvature on the discrete surfaces with triangular meshes, and considered the variation principle for the area and the energy of simplicial surfaces with triangulations. There are many related works as in \cite{Hildebrandt(2006), Lam(2018), Polthier(2005)}, etc. Similarly, for discrete CMC surfaces with triangular meshes, Oberknapp, Polthier and Rossman studied in the scene of the variational approaches, in \cite{Oberknapp(1997), PolthierRossman(2002)}. Recently, in \cite{KNO(2017)} Kotani, Naito and Omori established a new approach to study discrete surfaces on a $3$-valent graph. They defined the mean and Gaussian curvatures for the discrete surfaces on  
a $3$-valent graph, which have the triangular tangent plane. Their motivation is same as this paper, and they considered the application for atomic configurations of materials. However, their domain was a ``$3$-valent graph'' and they assumed their discrete surface as an ``embedding''. It contrasts with our situation on the full $3$-ary oriented tree.     

In this paper, firstly we will study discrete surfaces in $\mathbb{R}^3$ on a full $3$-ary oriented tree, and our theory extends the results in \cite{KNO(2017)} from a $3$-valent graph to a full $3$-ary oriented tree. Moreover, our geometric object is a map including an ``immersion'' which differs from an embedding in \cite{KNO(2017)}. Secondly, we will define a discrete analogue of principal directions, called {\it discrete principal directions}. Using this notion, we study discrete constant principal curvature (CPC)  surfaces, and construct some examples (discrete round cylinders and CPC tori). 

We will show the structure of the paper. In Section \ref{sec:dicrete surface}, we define the discrete surfaces as maps of a full $3$-ary oriented tree to $\mathbb{R}^3$. In Remark \ref{rema:covering}, we also consider the difference between our definition and one of \cite{KNO(2017)}, and geometric meaning for them. In Section \ref{sec:principal}, we define the {\it discrete principal directions} as a discrete analogue of the principal directions. Theorem \ref{theo:1} implies geometric meanings of the discrete principal directions, and we can get the center of curvature spheres on the discrete surfaces, and characterization of discrete CPC surfaces, via the parallel transformations. By Theorem \ref{theo:2}, the existence of discrete surfaces in the both discrete principal directions, are guaranteed, even if we assume the strong restriction that they should have constant bond length. In Section \ref{sec:examples}, we introduce some examples of discrete surfaces, which has the property of the discrete principal directions (see Theorem \ref{theo:3} and \ref{theo:4}). In Section \ref{sec:future}, we discuss our future directions as a next project. In this paper, we treat ``constant bond length'', but for general we may study about the case of non-constant bond length. However, there already have existed the complicated examples of CPC surfaces with non-constant bond length in our previous paper \cite{KMNOO(??)}, and they have a quite different geometric property from smooth cases, related to the {\it center-axisoid} (see Remark \ref{rema:CPCsmooth} in this paper, and \cite{KMNOO(??)} for detail).

\section{Review of constant principal curvature surfaces}\label{sec:smooth}
In this section, we review the classical theory of constant principal curvature (CPC) surfaces  in $\mathbb{R}^3$. They are also called the tubular surfaces, pipe surfaces or canal surface, as introduced in \cite{Anciaux(2015), Garcia(2006), Hilbert(1952)}, etc. Let $\Sigma$ be a simply connected domain, and let $f:\Sigma \to \mathbb{R}^3$ be a regular surface with curvature line coordinates $(u,v)$ as follows:
$$\langle f_u, f_v\rangle=0,\  \nu_u=-k_1f_u\ \textup{and}\ \nu_v=-k_2f_v$$
for the standard Euclidean inner product $\langle \cdot, \cdot \rangle: \mathbb{R}^3\times\mathbb{R}^3\to \mathbb{R}$, the unit normal vector $\nu\in \mathbb{S}^2$, and the principal curvatures $k_1$, $k_2$. Then, we have the following classical fact.
\begin{fact}[CPC surfaces]
An umbilic-free regular surface $f:\Sigma \to \mathbb{R}^3$ with curvature line coordinates $(u,v)$ has constant principal curvature $k_1\neq 0$ if and only if 
\begin{align}
f(u,v)=\gamma(v)-\frac{1}{k_1}\left(\cos \theta(u) b_1(v)+\sin \theta(u) b_2(v)\right),\label{eq:smoothCPC}
\end{align}
where $\theta(u)$ is a function in $u$, $\gamma(v)$ is a regular space curve, and $b_1(v), b_2(v)$ are orthonomal basis in the normal space of $\gamma(t)$, i.e. $b_1(v), b_2(v)\in (\gamma_v(v))^\perp$. For the regularity of $f$ of the form \eqref{eq:smoothCPC}, i.e. $f$ does not have singularities on $\Sigma$, $\gamma$ should satisfy 
$$\theta_u\cdot ||\gamma_v-\frac{1}{k_1}\left(\cos \theta (b_1)_v+\sin \theta (b_2)_v\right)||\neq 0,$$ 
where $||\cdot||$ means the standard Euclidean norm in $\mathbb{R}^3$. 
\end{fact}
\begin{proof}
The necessary condition is trivial. We will show the sufficient condition. We assume $f$ is a CPC surface. Then, the ODE $\nu_u=-k_1f_u$ implies that $f(u,v)=\gamma(v)-\frac{1}{k_1}\nu(u,v)$ for a space curve $\gamma$. Using $\nu_v=-k_2f_v$, we have $\gamma_v=\left(1-\frac{k_2}{k_1}\right)f_v$. If $\gamma_v=0$, then $k_1=k_2$ or $f_v=0$. However, by assumptions, both cases do not occur, and thus $\gamma$ should be regular. Moreover, the fact $\gamma_v$ is parallel to $f_v$ implies $\nu$ is included in the span of $b_1$ and $b_2$. Consequently, we have \eqref{eq:smoothCPC}. It is also easy to show that the regularity conditions $||f_u||\neq0$ and $||f_v||\neq0$ equal to the above condition.
\end{proof}
\begin{coro}
Let $f$ be an umbilic-free regular surface with curvature line coordinates $(u,v)$ having constant principal curvature $k_1\neq 0$, given by a space curve $\gamma$. If $\gamma$ is a line, then $f$ becomes a round cylinder.
\end{coro}
In sections $3$ and $4$, we will study the discrete version of CPC surfaces. In the section \ref{sec:examples}, we also construct some concrete examples of discrete CPC surfaces.    

\section{Discrete surfaces on a full $3$-ary oriented tree}\label{sec:dicrete surface}
In this section, we extend the definition of discrete surfaces in $\mathbb{R}^3$ as in \cite{KNO(2017), KotaniNaitoTao(2022)}, which are embeddings of a trivalent oriented graph. We will consider a map of a full $3$-ary (ternary) oriented tree which is defined as a rooted tree in which each node has either $0$ or $3$  children.

\begin{defi}[c.f.\cite{KNO(2017)}]
Let $X=(V_X, E_X)$ be a full $3$-ary oriented tree in $\mathbb{R}^2$; $V_X$ is a set of vertices (root, node and leaf) and $E_X$ is a set of oriented edges. For a root or node $x\in V_X$, $x$ has three edges $E_x=\{e_1, e_2, e_3\}\subset E_X$.
\quad A map $\iota: X \to \mathbb{R}^3$ is said to be a {\it discrete surface} if 
at least two vectors in $\{\iota(e)\ |\ e\in E_x\}$ are linearly independent for each $x\in V$, and 
$\iota(x)$ is locally oriented, that is, the order of the three edges is assumed to be assigned to each vertex of $X$ (see Figure \ref{fig:iota}).
\begin{figure}[H]
\begin{center}
\includegraphics[width=0.75\linewidth]{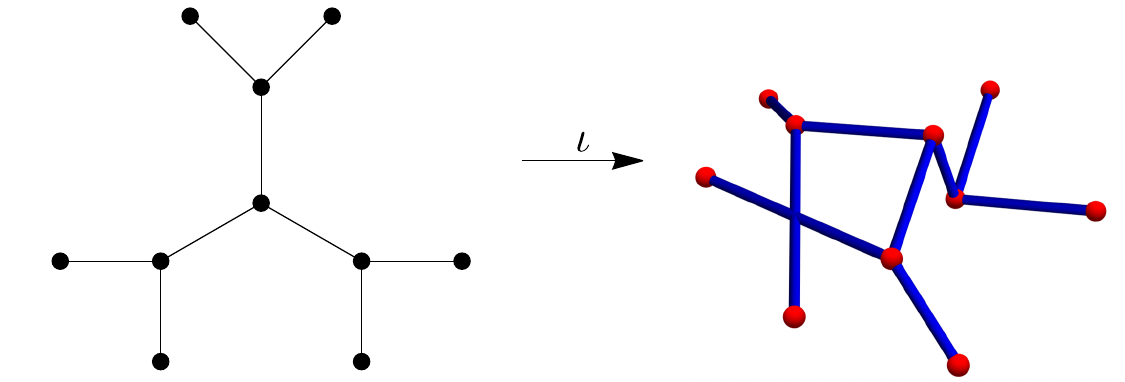}
\end{center}
\caption{Definition of $\iota(X)$.}\label{fig:iota}
\end{figure}
\end{defi}
\begin{rema}\label{rema:covering}
\begin{itemize}
\item[(1)] We refer to the map $\iota(x)$ constructing a discrete surface as an "immersion."
\item[(2)]  In the above definition, we remark that ours differ from \cite{KNO(2017)}. We replace an embedding $\iota(X)$ in \cite{KNO(2017)} with a general map, and we employ the full $3$-ary oriented trees rather than trivalent oriented graphs as in \cite{KNO(2017)}. We can regard a full $3$-ary oriented tree $X$ as a multiple cover of a trivalent graph $\tilde{X}$, i.e. a covering map $c:X\to \tilde{X}$ exists. However, even for a map $\iota:X \to \mathbb{R}^3$, it is not guaranteed that there is an embedding $\tilde{\iota}:\tilde{X} \to \mathbb{R}^3$ such that $\iota=\tilde{\iota}\circ c: X\to \mathbb{R}^3$. For later convenience, we employ the definition (see Figure \ref{fig:covering}). 
\item[(3)] We also remark that we can find a subgraph $X$ in a trivalent connected graph or its covering trivalent graph so that it is a full $3$-ary oriented tree.
\end{itemize}
\begin{figure}[H]
\begin{center}
\includegraphics[width=0.6\linewidth]{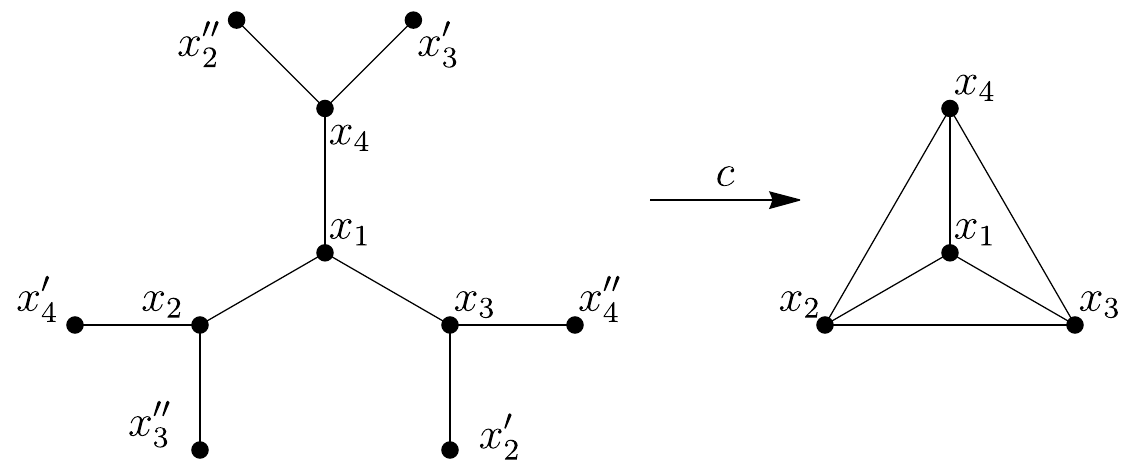}
\end{center}
\caption{Example of covering map $c$. In the above, the map $c$ satisfies $c(x_1)=x_1$ and $c(x_i)=c(x'_i)=c(x''_i)=x_i$ for $i=2,3,4$.}\label{fig:covering}
\end{figure}
\end{rema}
Let $\iota: X \to \mathbb{R}^3$ be a discrete surface on a full $3$-ary oriented tree $X$. Fix a root or node $x\in V_X$ and set $x_i\ (i=1,2,3)$ as adjacent vertices with $x$. Then, the oriented edges at $\iota(x)$ are determined as $\iota(e_i)=\iota(x_i)-\iota(x)$ for $i=1,2,3$. As in \cite{KNO(2017)}, we call $v_1(x)=\iota(e_1)-\iota(e_3)=\iota(x_1)-\iota(x_3)$ and $v_2(x)=\iota(e_2)-\iota(e_3)=\iota(x_2)-\iota(x_3)$ as {\it tangent vectors}.  
\begin{defi}[\cite{KNO(2017)}]
Let $\iota: X \to \mathbb{R}^3$ be a discrete surface on a full $3$-ary oriented tree $X$. At the
root or node $x\in V_X$, we define the {\it unit normal vector} $n(x)$ as 
$$n(x):=\frac{v_1 \wedge v_2}{||v_1 \wedge v_2||}=\frac{\iota(e_1) \wedge \iota(e_2)+\iota(e_2) \wedge \iota(e_3)+\iota(e_3) \wedge \iota(e_1)}{||\iota(e_1) \wedge \iota(e_2)+\iota(e_2) \wedge \iota(e_3)+\iota(e_3) \wedge \iota(e_1)||},$$
where $||\cdot||$ is the standard Euclidean norm in $\mathbb{R}^3$ and $\wedge$ is the cross product. 
$n(x)$ is also a unit normal vector of the tangent plane $T_x:=\textup{span}\{v_1(x),v_2(x) \}.$\par
Let the standard Euclidean inner product be $\langle \cdot, \cdot \rangle: \mathbb{R}^3\times\mathbb{R}^3\to \mathbb{R}$. Then we define the {\it first fundamental form} $I$ and {\it second fundamental form} $I\hspace{-.2em}I$ as follows (see Figure \ref{fig:normal}):  
$$I=\left(\begin{array}{cc}
\langle v_1, v_1\rangle & \langle v_1, v_2\rangle\\
\langle v_2, v_1\rangle & \langle v_2, v_2\rangle
\end{array}\right),\ I\hspace{-.2em}I=-\left(\begin{array}{cc}
\langle v_1, n(x_1)-n(x_3)\rangle & \langle v_1, n(x_2)-n(x_3)\rangle\\
\langle v_2, n(x_1)-n(x_3)\rangle & \langle v_2, n(x_2)-n(x_3)\rangle
\end{array}
\right).$$
As in \cite{KNO(2017)}, the {\it discrete mean curvature} $H(x)$ and {\it Gaussian curvature} $K(x)$ are defined as 
$$H(x)=\mathrm{tr}(I^{-1}I\hspace{-.2em}I),\quad \textup{and}\quad K(x)=\mathrm{det}(I^{-1}I\hspace{-.2em}I).$$
In this paper, we also define the {\it principal curvatures} $k_i(x)$ ($i=1,2$) by  
$$\quad  k_i(x)=\frac{H(x)\pm\sqrt{H(x)^2-4K(x)}}{2},$$
where the signs are defined as $+$ (resp. $-$) when $i=1$ (resp. $i=2$).
\end{defi}

\begin{figure}[H]
\begin{center}
\includegraphics[width=0.85\linewidth]{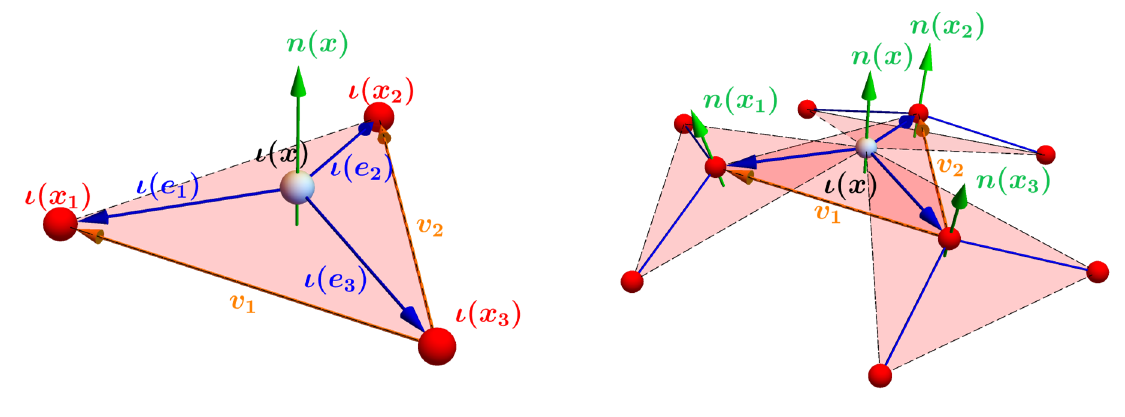}
\end{center}
\caption{Unit normal vectors on $\iota(X)$.}\label{fig:normal}
\end{figure}

\begin{rema}
We notice that the definition of $n(x)$ does not depend on the change of labels $i=1,2,3$, up to the plus and minus signs. By Proposition 3.5, 3.6 and Remark 3.7 in \cite{KNO(2017)}, we also notice that $H(x)$ and $k_i(x)$ have a same property, and $K(x)$ is same even when $n(x)$ changes $n(x)\mapsto -n(x)$. 
\end{rema}
\section{Discrete principal directions}\label{sec:principal}
Here we will define a discrete analogue of curvature lines on the discrete surface $\iota(X)$. If we consider the curvature line coordinates on a smooth surface, they have many good properties related to principal curvatures; Rodriguez formula, diagonalization of $I,\ I\hspace{-.2em}I$, etc.
However, in the discrete surface theory, we can not obtain whole of them, and the differences of the unit normal vectors are denoted by 
$$dn(x)=n(x_1)-n(x_3)\quad \textup{and}\quad d'n(x)=n(x_2)-n(x_3)$$
for short. 
\begin{lemm}\label{lemm:general I,II}
For a discrete surface $\iota(X)$ on a full $3$-ary oriented tree $X$, assume the unit normal vector satisfies 
$$dn(x)=\alpha_1(x)v_1(x)+\alpha_2(x)v_2(x)+\alpha_3(x)n(x),\quad d'n(x)=\beta_1(x)v_1(x)+\beta_2(x)v_2(x)+\beta_3(x)n(x)$$
for some real-valued functions $\alpha_i$ and $\beta_i$ ($i=1,2,3$). Then, the Weingarten map $I^{-1}I\hspace{-.2em}I$ is 
$$I^{-1}I\hspace{-.2em}I=\left(\begin{array}{cc}
-\alpha_1(x) & -\beta_1(x)\\
-\alpha_2(x) & -\beta_2(x)
\end{array}
\right).$$
Moreover, if $\alpha_2(x)\beta_1(x)=0$, then the principal curvatures equal $\alpha_1(x)$ and $\beta_2(x)$.
\end{lemm}
\begin{proof}
Just calculation.
\end{proof}
Considering the above lemma as an important clue, we will define discrete principal directions as follows:
\begin{defi}[Discrete principal directions]
Let $\iota: X \to \mathbb{R}^3$ be a discrete surface on a full $3$-ary oriented tree $X$. For a fixed point $x\in V_X$, at the root or node $\iota(x)$, we call $v_1(x)$ (resp. $v_2(x)$) as the {\it discrete principal $k_1(x)$-direction} (resp. $k_2(x)$-direction) if
$$dn(x)=-k_1(x)v_1(x)\quad (\textup{resp.}\ d'n(x)=-k_2(x)v_2(x))$$
 for the principal curvatures $k_1(x),\ k_2(x)$, the difference of unit normal vectors $dn(x)=n(x_1)-n(x_3)$ and $d'n(x)=n(x_2)-n(x_3)$. 
\end{defi}
\begin{rema}
The definition of discrete principal directions is an analogue of the Rodriguez equations for a smooth surface $f(u,v)\in \mathbb{R}^3$ with curvature line coordinates $(u,v)$:
$$\nu_u=-k_1f_u,\ \nu_v=-k_2f_v$$
for the unit normal vector $\nu(u,v)$.  
\end{rema}  
Here we introduce some properties of discrete principal directions.
\begin{coro}
Let $\iota: X \to \mathbb{R}^3$ be a discrete surface on a full $3$-ary oriented tree $X$. At the root or node $\iota(x)$, assume $v_i(x)$ ($i=1,2$) as the both discrete principal $k_i(x)$-directions. Then, the Weingarten map $I^{-1}I\hspace{-.2em}I$ is diagonalized as  
$$I^{-1}I\hspace{-.2em}I=\left(\begin{array}{cc}
k_1(x) & 0\\
0 & k_2(x)
\end{array}
\right).$$
\end{coro}
\begin{proof}
We get the assertion by Lemma \ref{lemm:general I,II} with $\alpha_1=-k_1,\ \alpha_2=\alpha_3=\beta_1=\beta_3=0,\ \beta_2=-k_2$. 
\end{proof}
\begin{rema}
We notice that $I$ and $I\hspace{-.2em}I$ are not diagonalized in general, even if we assume $v_i(x)$ ($i=1,2$) as the both discrete principal $k_i(x)$-directions. 
\end{rema}
We introduce one of our main theorem related to the discrete principal directions and the center of curvature spheres. Before the theorem, we define the parallel transformation $P(t,x)$ of $x$.
\begin{defi}
Let $\iota: X \to \mathbb{R}^3$ be a discrete surface on a full $3$-ary oriented tree $X$. At any root or node $\iota(x)$, set $n(x)$ be the unit normal vector of $\iota(x)$. We define $P(t,x):=\iota(x)+tn(x)$ as the {\it parallel transformation} of $\iota(x)$, with a parameter $t\in \mathbb{R}$. 
\end{defi}
\begin{theo}[discrete principal directions and the center of curvature spheres]\label{theo:1}\ \par
Let $\iota: X \to \mathbb{R}^3$ be a discrete surface on a full $3$-ary oriented tree $X$. At any root or node $\iota(x)$, set $x_i\ (i=1,2,3)$ as adjacent vertices with $x$, and assume $v_1(x)$ as the discrete principal $k_1(x)$-direction. Define $P(t,x)$ as the parallel transformation with a parameter $t\in \mathbb{R}$.
\begin{itemize}
\item[(1)] If $t=\frac{1}{k_1(x)}$ for $k_1(x)\neq 0$, two points $P\left(\frac{1}{k_1(x)},x_1\right)$ and $P\left(\frac{1}{k_1(x)},x_3\right)$ degenerate to a same point $p$. We can regard $p$ as a center of the curvature sphere in the $k_1(x)$-direction.
\item[(2)] Let $d'n(x)=\beta_1(x)v_1(x)+\beta_2(x)v_2(x)+\beta_3(x)n(x)$. If $\iota(x)$ is a discrete constant principal curvature $k_1$ (CPC $k_1$) surface with $\beta_3(x)=0$, then for any $t$, $P(t,x)$ also becomes a discrete CPC $\frac{k_1}{1-tk_1}$ with the discrete principal $k_1(x)$-direction $v_1$. Moreover, $P(t,x)$ also has its unit normal vector as $n(P(t,x))=n(x)$.
\end{itemize}
\end{theo}
\begin{figure}[H]
\begin{center}
\includegraphics[width=0.6\linewidth]{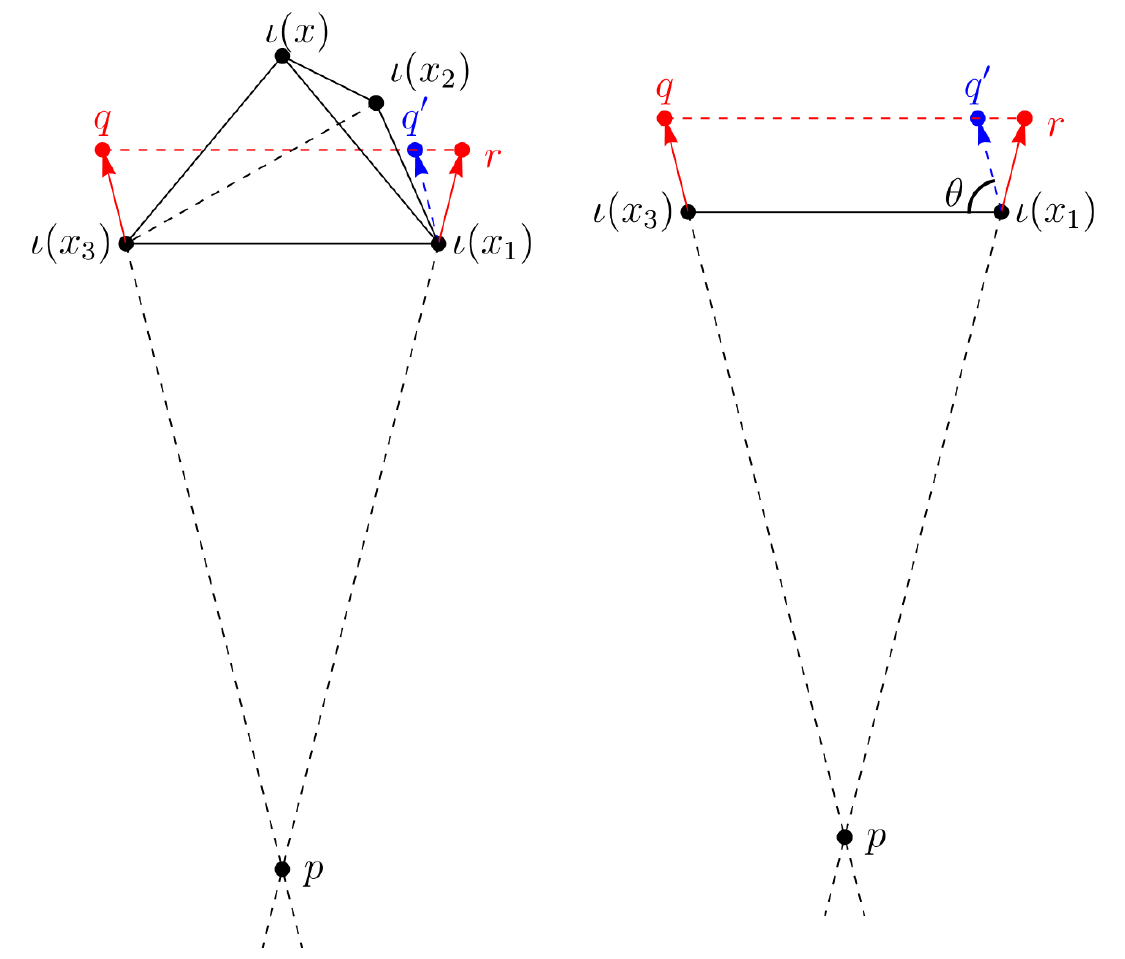}
\end{center}
\caption{Parallel transformation of $\iota(x)$.}\label{fig:parallel}
\end{figure}

\begin{proof}
(1) As in Figure \ref{fig:parallel} (left), set the vertices $q=\iota(x_3)+n(x_3)$ and $r=\iota(x_1)+n(x_1)$. Since $v_1(x)$ is the discrete principal $k_1(x)$-direction, $q,\ r,\ \iota(x_1)$ and $\iota(x_3)$ lie on the same plane, and we have $r-q\neq \iota(x_3)-\iota(x_1)$ for $k_1(x)\neq 0$. Thus, there exists an intersection $p$ between lines $q\iota(x_3)$ and $r\iota(x_1)$.\par
Next we take a point $q'$ as $q'=q-\iota(x_3)+\iota(x_1)$. Then, $||q'-r||=|k_1(x)|\cdot ||\iota(x_1)-\iota(x_3)||$ and $||q'-q||=||\iota(x_1)-\iota(x_3)||$ hold. By $\triangle p\iota(x_1)\iota(x_3)\sim \triangle pqr$, we have
$$p=\iota(x_3)+\frac{1}{k_1(x)}n(x_3)=\iota(x_1)+\frac{1}{k_1(x)}n(x_1).$$ 
(2) By assumptions, we have $v_1(P(t,x))=v_1(x)+t dn(x)=(1-tk_1)v_1(x)$ and $v_2(P(t,x))=v_2(x)+t d'n(x)=t\beta_1(x)v_1(x)+(1+t\beta_2(x))v_2(x)$. Thus, $n(x)=n(P(t,x))$ and $dn(P(t,x))=dn(x)=-k_1v_1(x)=-\frac{k_1}{1-tk_1}v_1(P(t,x)).$
\end{proof}
\begin{coro}
Let $\iota: X \to \mathbb{R}^3$ be a discrete surface on a full $3$-ary oriented tree $X$. At the root or node $\iota(x)$, set $x_i\ (i=1,2,3)$ as adjacent vertices with $x$, and assume $v_1(x)$ as the discrete principal $k_1(x)$-direction. Then, 
$$k_1(x)=\frac{2\langle v_1(x), n(x_3)\rangle}{||v_1(x)||^2}.$$
This formula implies $k_1(x)$ does not depend on $n(x_1)$.
\end{coro}
\begin{proof}
We can consider same setting with the proof of (1) in the above theorem. As in Figure \ref{fig:parallel} (right), set $\theta=\cos^{-1}\left(\frac{-\langle v_1(x),n(x_3)\rangle}{||v_1(x)||}\right)$. We also notice that $\angle \iota(x_3)qq'=\angle \iota(x_1)rq'=\angle \iota(x_1)q'r=\theta$ and $||q'-r||=|k_1(x)|\cdot ||v_1(x)||=2\cos \theta$ hold. Using these, we get the assertion.
\end{proof}
\begin{rema}\label{rema:CPCsmooth}
In the smooth surface theory of Section \ref{sec:smooth}, umbilic-free regular CPC $k_1$ surfaces are called as the pipe surfaces, and they always converge to a regular space curve $\gamma$ in $\mathbb{R}^3$ via parallel transformations. However, by Theorem \ref{theo:1}, in the discrete setting a discrete CPC $k_1$ surface may not degenerate to a discrete curve (bivalent graph) in general, because there exist some examples (Figure \ref{fig:standardArmchair}) that discrete CPC $k_1$ surface degenerates to a discrete surface (called the {\it center-axisoid} in \cite{KMNOO(??)}). See also Sections \ref{sec:examples} and \ref{sec:future} in this paper. 
\end{rema}
Next we will show our second main theorem. It is related to the existence of discrete surfaces in the both discrete principal $k_i(x)$-directions even when we assume the restriction of constant bond length, i.e. the length of edges are always same. For the application to nano material science \cite{KMNOO(??)}, it is natural and important to study the discrete surface theory under the setting of constant bond length. 
\begin{theo}[Existence of discrete surfaces with the discrete principal directions and constant bond length]\label{theo:2}
Let $r\in \mathbb{R}\backslash\{0\}$ be a constant, and $X$ be a full $3$-ary oriented tree. Under the given $k_1,\ k_2: V_X\to \mathbb{R}$, there exist at most $3$-parameter families of discrete surfaces $\iota(X)$ in the both discrete principal $k_i(x)$-directions and with the constant bond length $r$, up to the rigid motions and choices of $v_1(x),\ v_2(x)$ at every $x\in V_X$.
\end{theo}
\begin{proof}
We will show it as the following steps:\\
\underline{Step $1$:} Up to rigid motions, we can fix $\iota(x_0)$ and one of its adjacent vertex $\iota(x_3)$ uniquely.\\
\underline{Step $2$:} We define the unit normal vector $n(x_0)$ arbitrary, i.e. 
$n(x_0)=(\cos(\theta_1)\cos(\phi_1),\cos(\theta_1)\sin(\phi_1),$ $\sin(\theta_1))$, and up to the rotation with the $\iota(x_0)\iota(x_3)$-axis, without loss of generality, we can fix $\phi_1=0$. Similarly, we can define $n(x_3)=(\cos(\theta_2)\cos(\phi_2),\cos(\theta_2)\sin(\phi_2),\sin(\theta_2)).$\\
\underline{Step $3$:} Next we can define other adjacent vertices $\iota(x_1)$ and $\iota(x_2)$ uniquely, by the following restrictions.
\begin{itemize}
\item on the plane with the unit normal vector $n(x_0)$
\item on the sphere with the center $\iota(x_0)$ and radius $r$
\item satisfying the equations $k_1(x_0)=\frac{2\langle v_1(x_0), n(x_3)\rangle}{||v_1(x_0)||^2}$ and $k_2(x_0)=\frac{2\langle v_2(x_0), n(x_3)\rangle}{||v_2(x_0)||^2}.$     
\end{itemize}
\underline{Step $4$:} Similarly, we can apply the same procedure of Step $3$ in order to determine the adjacent vertices $\iota(x_{31})$ and $\iota(x_{32})$ of $\iota(x_3)$.\\
\underline{Step $5$:} By the property of the discrete principal directions, $n(x_{31})$ and $n(x_{32})$ are uniquely determined.\\ 
\underline{Step $6$:} Apply the same procedure of Step $3$ again, in order to determine adjacent vertices of $\iota(x_{31})$, $\iota(x_{32})$, $\iota(x_{1})$ and $\iota(x_{2})$.\par
We can apply the above steps repeatedly to determine all $\iota(x)$, and only have at most $3$-parameters $\phi_1,\ \theta_2$ and $\phi_2$. 
\end{proof}

\section{Examples of discrete CPC surfaces}\label{sec:examples}
In this section, we construct some examples of discrete CPC surfaces.
Considering the results in \cite{KNO(2017)}, we see that some of regular polyhedrons (hexahedron and dodecahedron) become discrete surfaces in the both discrete principal $k_i$-directions and with constant bond length. They also showed a regular truncated icosahedron has the same property because of $\iota(x)=rn(x)$ for the principal curvature $k_1(x)=k_2(x)=-\frac{1}{r}$.
By Theorem \ref{theo:1} and Remark \ref{rema:CPCsmooth}, we will divide cases of discrete CPC $k_1$ surfaces considering the following geometric properties:
\begin{eqnarray}
\textup{the discrete principal $k_1$ (reps. $k_2$)-direction,}\label{eq:cond1}\\
\textup{$\beta_3(x)=0$ (resp. $\alpha_3(x)=0$),}\label{eq:cond2}\\
\textup{constant bond length,} \label{eq:cond3}
\end{eqnarray}
where $\alpha_3(x)$ and $\beta_3(x)$ satisfy $dn(x)=\alpha_1(x)v_1(x)+\alpha_2(x)v_2(x)+\alpha_3(x)n(x)$, $d'n(x)=\beta_1(x)v_1(x)+\beta_2(x)v_2(x)+\beta_3(x)n(x)$.\par

Here we demonstrate some examples satisfying the above properties more. Before the calculation, we show the equivalent condition of the discrete principal directions:
\begin{eqnarray}
&&dn(x)=-k_1(x)v_1(x)\ \textup{and}\ d'n(x)=-k_2(x)v_2(x)\nonumber\\
&\Longleftrightarrow& \begin{cases}
n(x_1)=-k_1(x)\iota(x_1)-k_2(x)\iota(x_3)+c(x)\\
n(x_2)=-k_1(x)\iota(x_3)-k_2(x)\iota(x_2)+c(x)\\
n(x_3)=-k_1(x)\iota(x_3)-k_2(x)\iota(x_3)+c(x)\label{eq:principal}
\end{cases}
\end{eqnarray}
for a vector $c(x)\in\mathbb{R}^3$. 
\subsection{Discrete round cylinder}
We consider a ``discrete round cylinder'' that is a discrete analogue of a round cylinder in $\mathbb{R}^3$ (i.e. all vertices $\iota(x)$ lie on the smooth round cylinder with the radius $r$) and has $k_1(x)\equiv -\frac{1}{r}$ and $k_2(x)\equiv 0$. This is a mathematical model of carbon nanotubes.\par  
In the theorem 4.7 of \cite{KNO(2017)}, the authors in \cite{KNO(2017)} considered discrete round cylinders constructed by regular hexagons with constant bond length and calculated their mean and Gaussian curvatures.
However, as they mentioned in the remark 4.8 of \cite{KNO(2017)}, we cannot get $H=-\frac{1}{2r}$ in general (see also Figure \ref{fig:standardArmchair}). This implies that $k_2=0$ but $k_1\neq-\frac{1}{r}$, and we have some gaps between  smooth round cylinders and discrete round cylinders.\par
\begin{figure}[H]
\begin{center}
\includegraphics[width=0.9\linewidth]{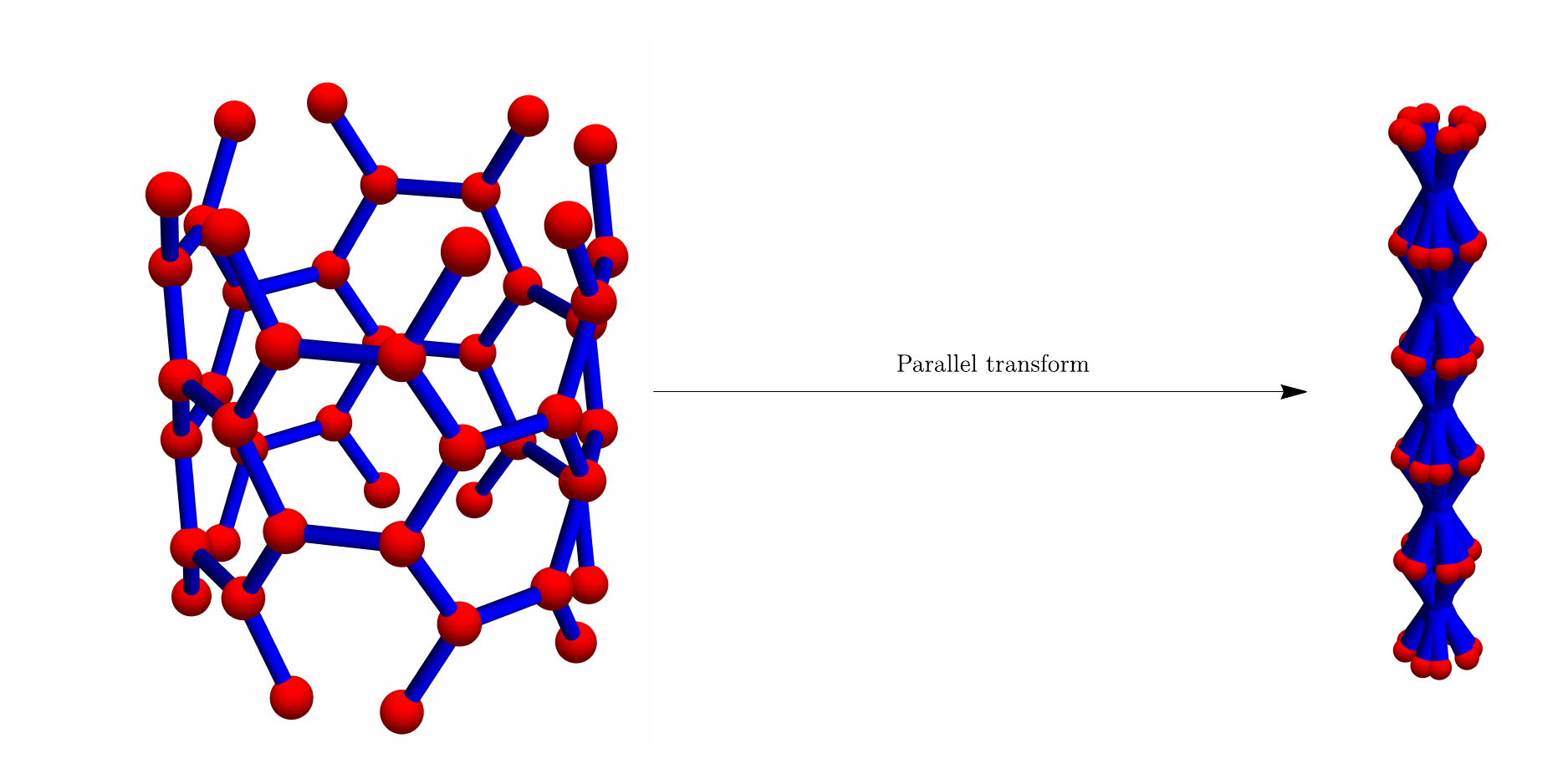}
\end{center}
\caption{Armchair-type discrete round cylinder in \cite{KNO(2017)} with constant bond length. It has the property in the discrete principal $k_2=0$-direction, but not in the $k_1$-direction. Moreover, $k_1\neq-\frac{1}{r}$ and this discrete round cylinder does not degenerate to a line via the parallel transformation.}\label{fig:standardArmchair}
\end{figure}
 To solve this problem, in \cite{KNO(2017)} they constructed armchair-type discrete round cylinders,
changing bond length (see also Figure \ref{fig:KNOtube}).
\begin{figure}[H]
\begin{center}
\includegraphics[width=0.9\linewidth]{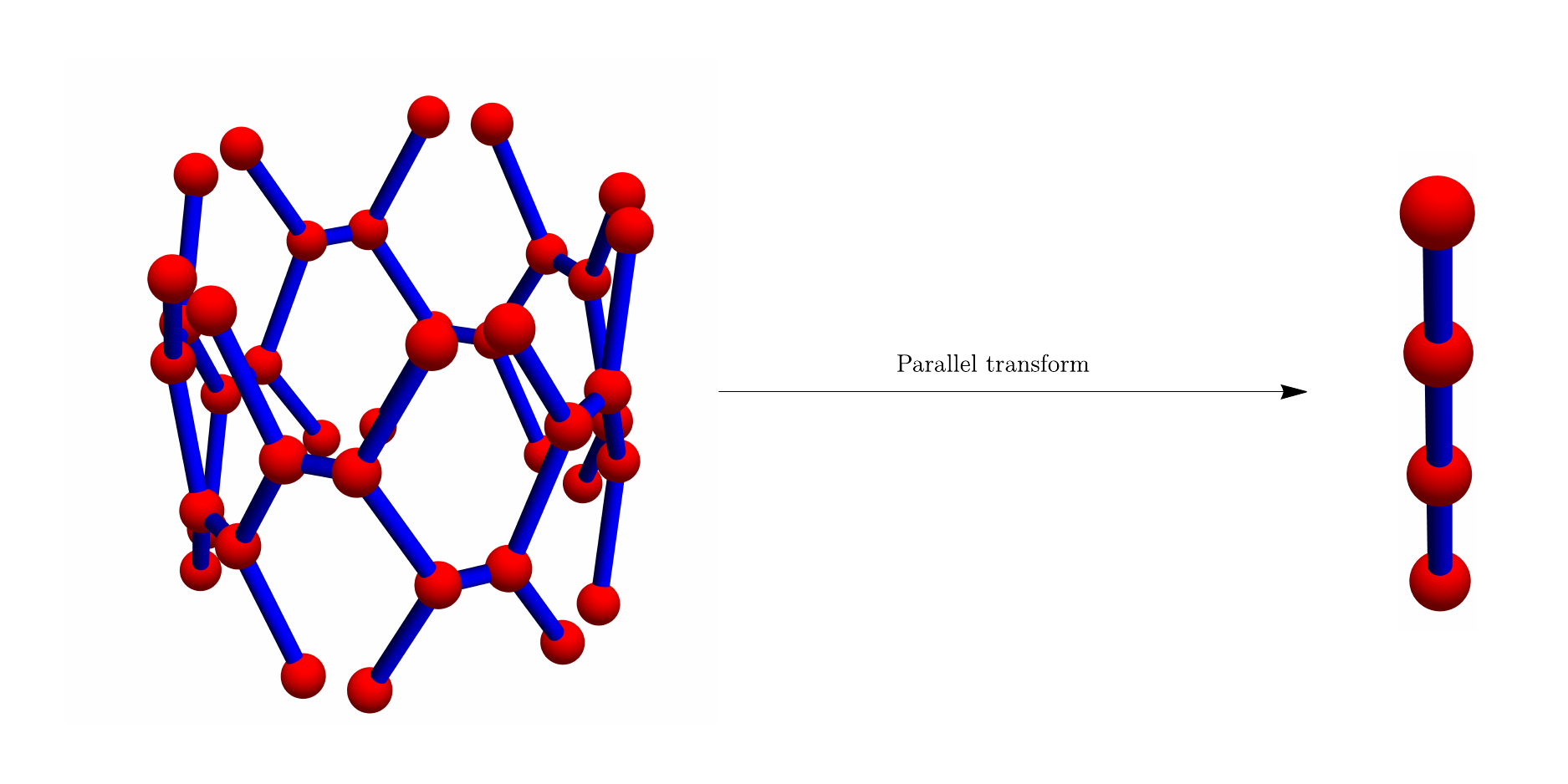}
\end{center}
\caption{Armchair-type discrete round cylinder in \cite{KNO(2017)} without constant bond length. It has the property in the discrete principal $k_2=0$-direction, but not in the $k_1=-\frac{1}{r}$- direction. We can also calculate $\alpha_3(x)=0$, and thus by Theorem \ref{theo:1} their discrete round cylinder also degenerates to a line via the parallel transformation.}\label{fig:KNOtube}
\end{figure}
In this paper, we also challenge to obtain discrete round cylinders constructed by regular hexagons with ``constant bond length, $k_1=-\frac{1}{r}$ and $k_2=0$'', analyzing the discrete principal directions. We assume their all faces are constructed by hexagons with constant bond length. By the equation \eqref{eq:principal} with $k_1(x)\equiv -\frac{1}{r}$ and $k_2(x)\equiv 0$, we have
\begin{eqnarray}
n(x_1)=\frac{1}{r}\iota(x_1)+c(x),\ 
n(x_2)=n(x_3)=\frac{1}{r}\iota(x_3)+c(x).\label{eq:principal2}
\end{eqnarray}
Now we show a solution of the above equation.
\begin{theo}\label{theo:3}
There exists a chiral-type discrete round cylinder having $k_1(x)\equiv -\frac{1}{r}$ and $k_2(x)\equiv 0$, in the both discrete principal $k_i$-directions and with constant bond length $d\in \mathbb{R}\backslash\{0\}$ (see also Figure \ref{fig:Chiral}).
\begin{figure}[H]
\begin{center}
\includegraphics[width=0.9\linewidth]{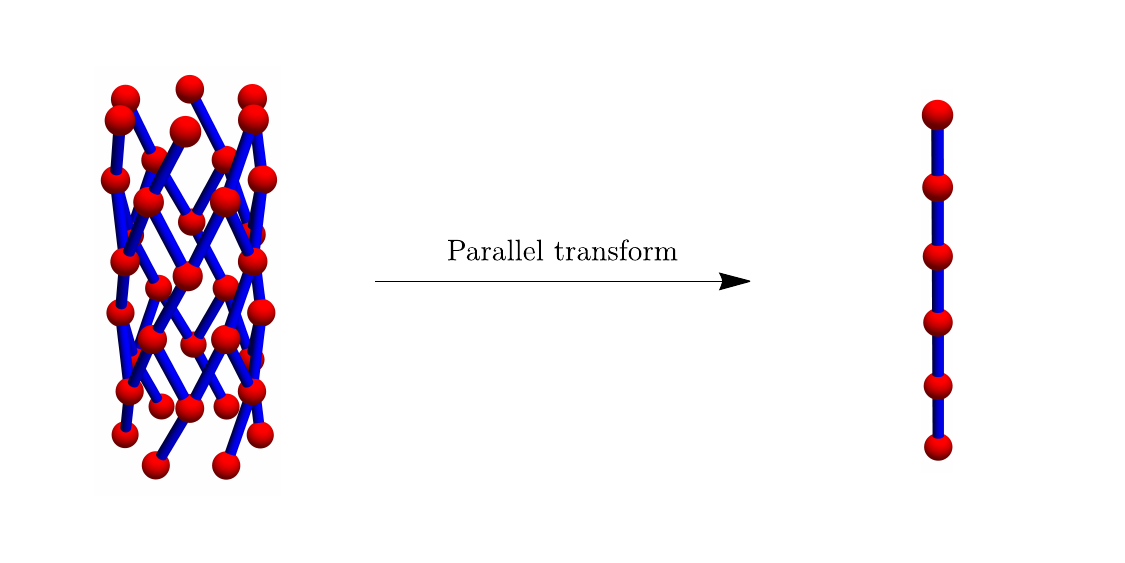}
\end{center}
\caption{Chiral-type discrete round cylinder with constant bond length, $k_1=-\frac{1}{r}$ and $k_2=0$.}\label{fig:Chiral}
\end{figure}
\end{theo}
\begin{proof}
We show the recipe. At a root $x_0\in V_x$, without loss of generality, we can set $c(x_0)=0$ in the equation \eqref{eq:principal2}. Define $\iota(x_0)=(r,0,0)$ and its adjacent vertices as $\iota(x_1)=(r\cos\theta, r\sin\theta,h),\ \iota(x_2)=(r\cos\theta, -r\sin\theta,-h)$ and $\iota(x_3)=(r\cos\theta, -r\sin\theta,h)$ for a constant $h\in\mathbb{R}\backslash\{0\}.$ Considering the equation \eqref{eq:principal2}, we also define $n(x_0)=\frac{1}{r}\iota(x_0),\ n(x_1)=\frac{1}{r}\iota(x_1)$ and $n(x_2)=n(x_3)=\frac{1}{r}\iota(x_3).$ Then, 
\begin{eqnarray*}
v_1(x_0)=(0,2r\sin\theta,0),\ v_2(x_0)=(0,0,-2h),\ 
dn(x_0)=\frac{1}{r}(0,2r\sin\theta,0),\ d'n(x_0)=\mathbf{0},\\
d=||\iota(x_0)-\iota(x_1)||=||\iota(x_0)-\iota(x_2)||=||\iota(x_0)-\iota(x_3)||=2r^2(1-\cos\theta)+h^2.
\end{eqnarray*}
For other vertices $x\in V_X$, we can apply the rotational and translational symmetry of this fundamental piece (i.e. a full $3$-ary oriented tree) $\{\iota(x_0), \iota(x_1), \iota(x_2), \iota(x_3)\}.$
\end{proof}
As a summary of this subsection, we have the following table \ref{table:1}.
\begin{table}[H]
\centering
    \begin{tabular}{|l||c|c|c|c|}  \hline
    & condition \eqref{eq:cond1} & $\alpha_3$ & $\beta_3$ & condition \eqref{eq:cond3} 
   \\ \hline \hline
   Fig. \ref{fig:standardArmchair} & $k_1$: $\times$, $k_2$: $\bigcirc$  & $\neq 0$ & $0$ & $\bigcirc$ \\ \hline 
   Fig. \ref{fig:KNOtube} & $k_1$: $\times$, $k_2$: $\bigcirc$  & $0$ & $0$ & $\times$  \\ \hline
   Fig. \ref{fig:Chiral}& $k_1$: $\bigcirc$, $k_2$: $\bigcirc$  & $0$ & $0$ & $\bigcirc$ \\ \hline
  \end{tabular}
\caption{Comparison between three models of discrete round cylinders. In this table, $\bigcirc$ means the condition satisfies, and $\times$ means not so.}\label{table:1} 
\end{table}

\subsection{Discrete CPC tori.}
Here we consider a discrete analogue of compact examples. Topologically spherical discrete surfaces are well-known, for example some of regular polyhedrons (hexahedron and dodecahedron) and regular truncated icosahedron. We now step to construct a discrete torus, especially a discrete CPC $k_1$ torus, that is a discrete analogue of a CPC $k_1$ torus in $\mathbb{R}^3$ (i.e. all vertices $\iota(x)$ lie on the smooth CPC $k_1$ torus, like a closed tube).
\begin{theo}\label{theo:4}
There exists a discrete CPC $k_1$ torus in the both discrete principal $k_i$-directions (see also Figure \ref{fig:Chiral}).
\end{theo}
\begin{figure}[H]
\begin{center}
\includegraphics[width=1.0\linewidth]{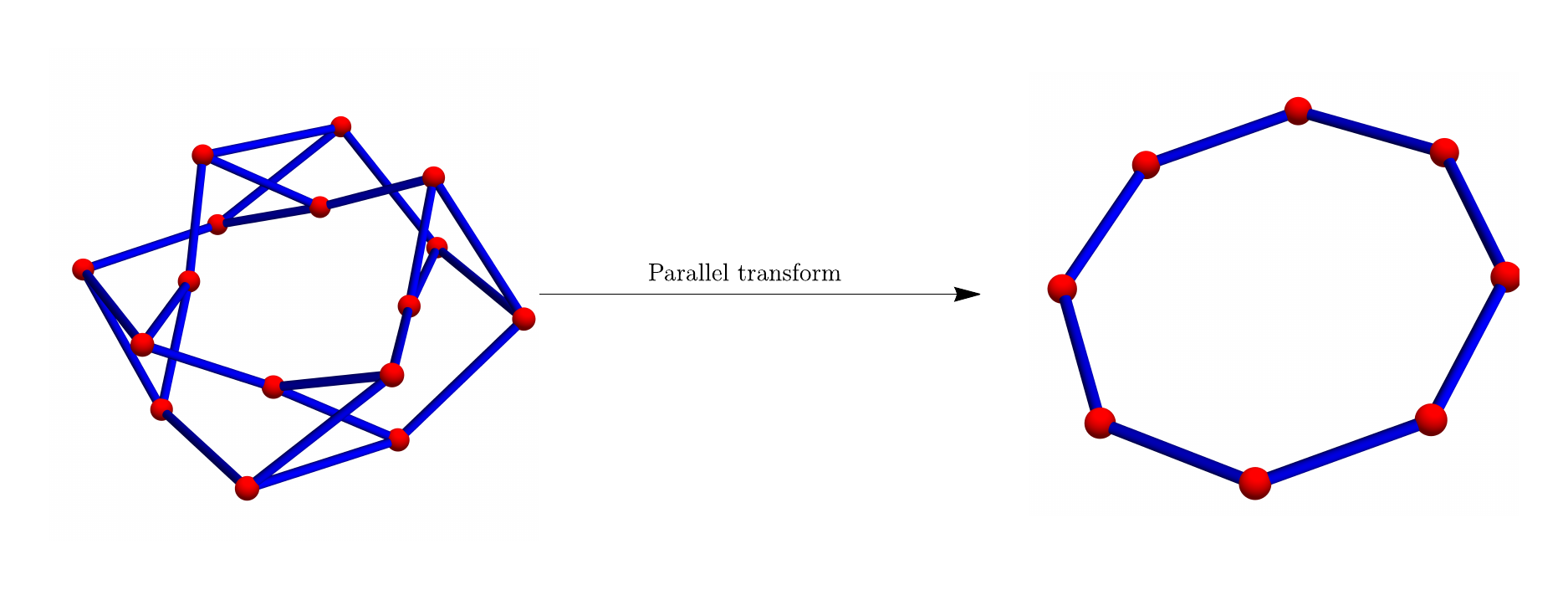}
\end{center}
\caption{Discrete CPC $1$ torus. To solve the closing conditions, we assumed all vertices $\iota(x)$ lies on $f(u,v)=((\cos u+3)\cos v, (\cos u+3)\sin v, \sin u)$ and set $p_0=\frac{\pi}{2}$, $p_1=0$, $p_2=\pi$ and $q_i=\frac{i\pi}{4}\ (i\in \mathbb{N})$ in the proof of Theorem \ref{theo:4}. Then, the resulting discrete surface has $k_1\equiv 1$ and in the both discrete principal $k_i$-directions.}\label{fig:CPCtorus}
\end{figure}
\begin{proof}
We show the recipe. Let $f(u,v)=((r_1\cos u+r_2)\cos v, (r_1\cos u+r_2)\sin v, r_1\sin u)$ be a smooth CPC torus in $\mathbb{R}^3$ for constants $r_1, r_2>0.$ Then, its unit normal vector $\nu$ is $\nu(u,v)=(\cos u\cos v, \cos u\sin v, \sin u)$. We also define vertices of a discrete CPC surface which lies on $f(u,v)$ as follows:
$$\iota(x)=f(p_0,q_0),\ \iota(x_3)=f(p_1,q_1),\ \iota(x_1)=f(p_1,q_2),\ \iota(x_2)=f(p_2,q_1),$$
where $p_1<p_0<p_2$ and $q_1<q_0<q_2.$ For the discrete principal directions, now we will consider the condition $n(x)=\nu(p_0,q_0)$. By the direct calculation, we have
\begin{eqnarray*}
v_1(x)=\iota(x_3)-\iota(x_1)=((r_1\cos p_1+r_2)(\cos q_1-\cos q_2),(r_1\cos p_1+r_2)(\sin q_1-\sin q_2),0)\ \textup{and}\\
v_2(x)=\iota(x_3)-\iota(x_2)=(r_1(\cos p_1-\cos p_2)\cos q_1,r_1(\cos p_1-\cos p_2)\sin q_1,r_1(\sin p_1-\sin p_2)),
\end{eqnarray*}
and thus $v_1(x)\times v_2(x)$ is parallel to 
$$((\sin p_1-\sin p_2)(\sin q_1-\sin q_2), -(\sin p_1-\sin p_2)(\cos q_1-\cos q_2), (\cos p_1-\cos p_2)\sin(q_1-q_2)).$$
By the above, $n(x)=\nu(p_0,q_0)$ equals to the following
\begin{eqnarray*}
&&
\begin{cases}
\tan q_0=-\displaystyle\frac{\cos q_1-\cos q_2}{\sin q_1-\sin q_2}\\
\displaystyle\frac{\tan p_0}{\cos q_0}=\displaystyle\frac{(\cos p_1-\cos p_2)\sin(q_1-q_2)}{(\sin p_1-\sin p_2)(\sin q_1-\sin q_2)}
\end{cases}\\
&\Longleftrightarrow& \therefore q_0=\frac{q_1+q_2}{2},\quad p_0=\mathrm{arctan}\left(\tan \frac{p_1+p_2}{2}\cos\frac{q_1-q_2}{2}\right).
\end{eqnarray*}
Consequently, we can achieve $n(x)=\nu(p_0,q_0)$. Using these conditions and rotational symmetry, we will get $dn(x)=-k_1v_1(x)=-\frac{1}{r_1}v_1(x)$ and $d'n(x)=-k_2(x)v_2(x)$ for some function $k_2(x)$. If we need to get a closed example, as one of simple choice, set $p_0=\frac{\pi}{2}$, $p_1=0$, $p_2=\pi$ and $q_i=\frac{i\pi}{N}\ (i, N\in \mathbb{N})$. 
\end{proof}
\section{Future directions}\label{sec:future}
In this paper, we introduced the discrete surface theory in $\mathbb{R}^3$ on a full $3$-ary oriented tree. We define the discrete principal directions on them as a discrete analogue of principal directions, and proved a relation with the center of curvatures (see Theorem \ref{theo:1}). Moreover, we also showed the existence of the discrete principal directions, even if we assume the condition of constant bond length (see Theorem \ref{theo:2}). In Section \ref{sec:examples}, we construct some interesting examples of discrete CPC surfaces by analyzing the discrete principal directions of them (see Theorem \ref{theo:3} and \ref{theo:4}).\par
   As a sequel study, we will need to analyze the cases of {\it non-constant bond length}. In the theory of nanocarbon materials, there exist small differences between bond lengths, via different types of the molecular bond. However, if we consider the setting of non-constant bond length for discrete surfaces, we have many freedoms, but some symmetries will be broken. For example, in Figure \ref{fig:FP5N}, the left is a new type of nanocarbon material, called a {\it peanut-shaped fullerene polymer} (PSFP) or a C${}_{60}$ polymer discovered in \cite{OnoeNakayamaAonoHara(2003), ShimaYoshiokaOnoe(2009)} and \cite{NodaOnoOhno(2015)}, which has non-constant bond length. 
In \cite{KMNOO(??)}, it is shown that it becomes a pre-CPC $k_1$ surface, i.e. we could regard it as a CPC $k_1$ surface except for a certain small deviation. However, as in Figure \ref{fig:FP5N}, it can not converge to an axis via parallel transforms, and there is a discrete surface, called the center-axisoid. We consider this phenomena as an important and interesting clue. Thus the relation between center-axisoids and the property of non-constant bond length should be investigated in the future.           
\begin{figure}[H]
\begin{center}
\includegraphics[width=1.0\linewidth]{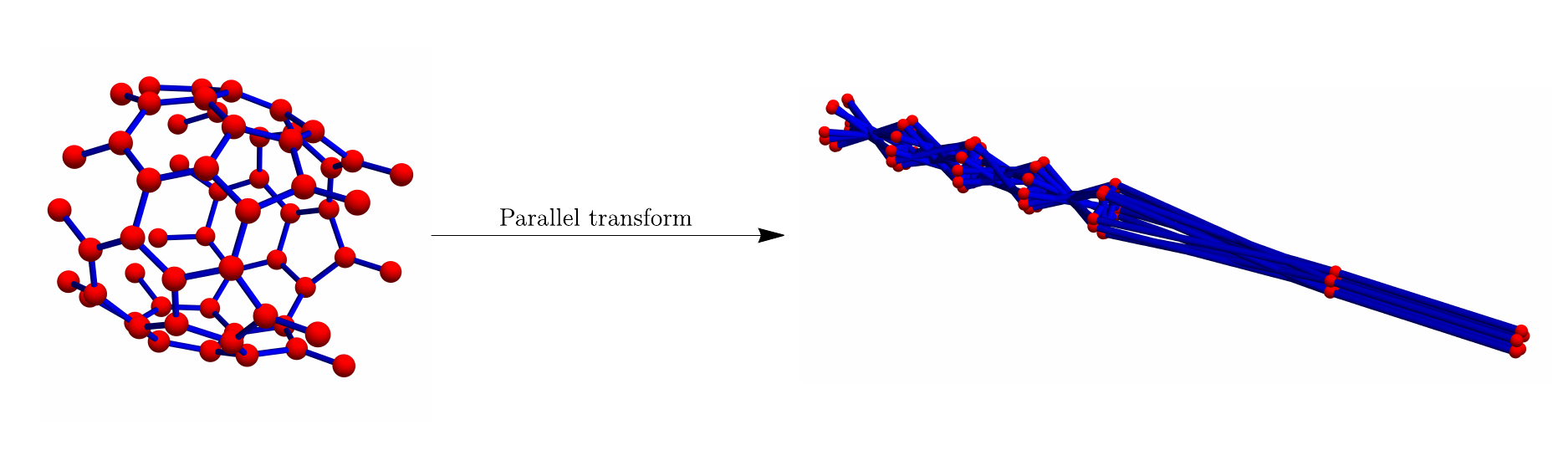}
\end{center}
\caption{Peanut-shaped fullerene polymer (PSFP) and its parallel transform; the figures show parts of periodic infinite graphs. }\label{fig:FP5N}
\end{figure}
\section*{Acknowledgement}
The authors are grateful to Professor Jun Onoe and Yusuke Noda for their valuable comments and for giving motivation for this study through collaboration \cite{KMNOO(??)}.
The first and third authors have been supported by the Grant-in-Aid for Young Scientist of Japan Society for the Promotion of Science Grant, no. 20K14312 and no. 21K13799, respectively. 
The second author has been supported by the Grant-in-Aid for Scientific Research (C) of Japan Society for the Promotion of Science Grant, no. 21K03289. This joint work was supported by Institute of Mathematics for Industry, Joint Usage/Research Center in Kyushu University. (“IMI workshop II: Geometry and Algebra in Material Science I”, September 7–8, 2020", (20200012), “IMI workshop II: Geometry and Algebra in Material Science II”, August 30-31, 2021", (20210001), and “IMI workshop I: Geometry and Algebra in Material Science III”, September 8–10, 2022", (2022a003)).
%
\bibliographystyle{abbrv}      
\bibliography{KMO1}

\end{document}